\documentclass[12pt, a4paper]{article}

\usepackage{amsfonts}
\usepackage{graphicx}

\usepackage[T1]{fontenc}
\usepackage{multirow}
\usepackage[centertags]{amsmath}
\usepackage{amsfonts}
\usepackage{amssymb}
\usepackage{amsthm}
 \newtheorem{thm}{Theorem}[section]
 \newtheorem{cor}[thm]{Corollary}
 \newtheorem{lem}[thm]{Lemma}
 \newtheorem{prop}[thm]{Proposition}
 \theoremstyle{definition}
 \newtheorem{defn}[thm]{Definition}
 \theoremstyle{remark}
 \newtheorem{rem}[thm]{Remark}
\theoremstyle{example}
\newtheorem{ex}[thm]{Example}
\theoremstyle{conjecture}

 \numberwithin{equation}{section}

\begin{document}
\title{Classification of quantum groups and Lie bialgebra structures on $sl(n,\mathbb{F}).$ \\
Relations with Brauer group}
\author{Alexander Stolin and Iulia Pop \footnote{Department of Mathematical Sciences,
University of Göteborg, 41296 Göteborg, Sweden. E-mail: astolin@chalmers.se;
iulia@chalmers.se. Corresponding author: Iulia Pop.}}

\date{}
\maketitle
\begin{abstract}

 Given an arbitrary field $\mathbb{F}$ of characteristic 0, we study Lie bialgebra structures on $sl(n,\mathbb{F})$, based on the description of the corresponding classical double.
For any Lie bialgebra structure $\delta$, the classical double
$D(sl(n,\mathbb{F}),\delta)$ is isomorphic to $sl(n,\mathbb{F})\otimes_{\mathbb{F}} A$,
where $A$ is either $\mathbb{F}[\varepsilon]$, with $\varepsilon^{2}=0$, or $\mathbb{F}\oplus\mathbb{F}$ or a quadratic field extension of $\mathbb{F}$. In the first case, the classification leads to quasi-Frobenius
Lie subalgebras of $sl(n,\mathbb{F})$. In the second and third cases, a Belavin--Drinfeld cohomology can be introduced which enables one to classify Lie bialgebras on $sl(n,\mathbb{F})$, up to gauge equivalence. The Belavin--Drinfeld untwisted and twisted cohomology sets associated to an $r$-matrix are computed. For the Cremmer--Gervais $r$-matrix in $sl(3)$, we also construct a natural map of sets
between the total Belavin--Drinfeld twisted cohomology set and the Brauer group of the field $\mathbb{F}$.


\end{abstract}

\section{Introduction}
Following \cite{D}, we recall that a quantized universal enveloping algebra (or a quantum group) over a field $k$ of characteristic zero is a topologically free topological Hopf algebra $H$ over the formal power series ring
$k[[\hbar]]$ such that $H/\hbar H$ is isomorphic to the universal enveloping algebra of a Lie algebra $\mathfrak{g}$ over $k$.

The quasi-classical limit of a quantum group is a Lie bialgebra. A Lie bialgebra is a Lie algebra $\mathfrak{g}$ together with a cobracket
$\delta$ which is compatible with the Lie bracket. Given a quantum group $H$,
with comultiplication $\Delta$, the quasi-classical limit of $H$ is the Lie bialgebra $\mathfrak{g}$ of primitive elements of $H/\hbar H$ and the cobracket is the restriction of the map $(\Delta-\Delta^{21})/\hbar (\mathrm{mod}\hbar)$ to
$\mathfrak{g}$.

The operation of taking the semiclassical limit is a functor
$SC: QUE\rightarrow LBA$ between categories of quantum groups and Lie bialgebras over $k$. The existence of universal quantization functors was proved by Etingof and Kazhdan \cite{EK1, EK2}. They used Drinfeld's theory of associators to construct quantization functors for any field $k$ of characteristic zero. More precisely, let $(\mathfrak{g},\delta)$ be a Lie bialgebra over $k$. Then one can associate a Lie bialgebra $\mathfrak{g}_{\hbar}$ over $k[[\hbar]]$
defined as $(\mathfrak{g}\otimes_{k}k[[\hbar]], \hbar\delta)$.
According to Theorem 2.1 of \cite{EK2}
there exists an equivalence $\widehat{Q}$
between the category $LBA_{0}(k[[\hbar]])$
of topologically free over $k[[\hbar]]$ Lie bialgebras
with $\delta=0 (\rm{mod}\hbar)$ and the category $HA_{0}(k[[\hbar]])$ of
topologically free Hopf algebras cocommutative modulo $\hbar$.
Moreover, for any $(\mathfrak{g},\delta)$ over $k$, one has the following:
$\widehat{Q}(\mathfrak{g}_{\hbar})=U_{\hbar}(\mathfrak{g})$.

Due to this equivalence, the classification of quantum groups whose quasi-classical limit is $\mathfrak{g}$ is equivalent to the classification of Lie bialgebra structures on $\mathfrak{g}\otimes_{\mathbb{C}}\mathbb{C}[[\hbar]]$.
Since any cobracket over $\mathbb{C}[[\hbar]]$ can be extended to one over
$\mathbb{C}((\hbar))$ and conversely, any cobracket over $\mathbb{C}((\hbar))$, multiplied by an appropriate power of $\hbar$, can be restricted to a cobracket over  $\mathbb{C}[[\hbar]]$,  this in turn reduces to the problem of finding Lie bialgebras on $\mathfrak{g}\otimes_{\mathbb{C}}\mathbb{C}((\hbar))$. Denote, for the sake of simplicity,  $\mathbb{K}:=\mathbb{C}((\hbar))$ and $\mathfrak{g}(\mathbb{K}):= \mathfrak{g}\otimes_{\mathbb{C}} \mathbb{K}$.

As a first step towards classification, following ideas of \cite{MSZ}, we proved in \cite{SP1} that for any Lie bialgebra structure on $\mathfrak{g}(\mathbb{K})$, the associated classical double is of the form $\mathfrak{g}(\mathbb{K})\otimes_{\mathbb{K}} A$, where $A$ is one of the following
associative algebras: $\mathbb{K}[\varepsilon]$, where $\varepsilon^{2}=0$, $\mathbb{K}\oplus\mathbb{K}$ or $\mathbb{K}[j]$, where $j^{2}=\hbar$.

As it was shown in \cite{SP1}, the classification of Lie bialgebras with classical double $\mathfrak{g}(\mathbb{K}[\varepsilon])$ leads to the classification of quasi-Frobenius Lie algebras over $\mathbb{K}$, which is a complicated and still open problem.

Unlike this case, the classification of Lie bialgebras with classical double
$\mathfrak{g}(\mathbb{K})\oplus \mathfrak{g}(\mathbb{K})$ can be achieved by cohomological and combinatorial methods. In \cite{SP1}, we introduced a Belavin--Drinfeld cohomology theory which proved to be useful for the study of Lie bialgebra structures.
To any
non-skewsymmetric $r$-matrix $r_{BD}$ from the Belavin--Drinfeld list \cite{BD}, we associated a cohomology set $H^{1}_{BD}(\mathfrak{g},r_{BD})$.
 We proved the existence of a one-to-one correspondence between any Belavin--Drinfeld cohomology and gauge equivalence classes of Lie bialgebra structures on $\mathfrak{g}(\mathbb{K})$. In case
$\mathfrak{g}=sl(n)$, we showed that for any
non-skewsymmetric $r$-matrix $r_{BD}$, the cohomology set $H^{1}_{BD}(sl(n),r_{BD})$ has only one class, which is represented by the identity.

Regarding the classification of Lie bialgebras whose classical double is isomorphic to $\mathfrak{g}(\mathbb{K}[j])$, with $j^{2}=\hbar$,
a cohomology theory can be introduced too. Our result states that there exists a one-to-one correspondence
between Belavin--Drinfeld
twisted cohomology and gauge equivalence classes of Lie
bialgebra structures on $\mathfrak{g}(\mathbb{K})$ whose classical double is isomorphic to $\mathfrak{g}(\mathbb{K}[j])$.
In \cite{SP1}, we proved that the twisted cohomology corresponding to the Drinfeld--Jimbo $r$-matrix has only one class, represented by a certain matrix $J$ (not the identity). A deeper investigation was done in the subsequent article \cite{SP2} where twisted cohomologies for $sl(n)$ associated to
generalized Cremmer--Gervais $r$-matrices were studied.

The aim of the present article is the study of Lie bialgebra structures on $sl(n,\mathbb{F})$, for an arbitrary field $\mathbb{F}$ of characteristic zero. Again the idea is to use the description of the classical double. We will show that for any Lie bialgebra structure $\delta$, the classical double
$D(sl(n,\mathbb{F}),\delta)$ is isomorphic to $sl(n,\mathbb{F})\otimes_{\mathbb{F}} A$,
where $A$ is either $\mathbb{F}[\varepsilon]$, with $\varepsilon^{2}=0$, or $\mathbb{F}\oplus\mathbb{F}$ or a quadratic extension of $\mathbb{F}$. In the first case, the classification leads to quasi-Frobenius
Lie subalgebras of $sl(n,\mathbb{F})$. In the second and third cases, we will introduce a Belavin--Drinfeld cohomology which enables one to classify Lie bialgebras on $sl(n,\mathbb{F})$, up to gauge equivalence. In the particular case
$\mathbb{F}=\mathbb{C}((\hbar))$ we recover the classification
of quantum groups whose classical limit is $sl(n,\mathbb{C})$ obtained
in \cite{SP1, SP2}.

\section{Description of the classical double}
From the general theory of Lie bialgebras it is known that for each Lie bialgebra structure $\delta$ on a fixed Lie algebra $L$ one can construct the corresponding classical double $D(L, \delta)$. As a vector space,  $D(L,\delta)=L\oplus L^{*}$.  Moreover, since the cobracket of $L$ induces a Lie bracket
on $L^{*}$, there exists a Lie algebra structure on $L\oplus L^{*}$, induced by the bracket and cobracket of $L$, and such that the canonical symmetric nondegenerate bilinear form $Q$ on this space is invariant.

Let $\mathbb{F}$ be an arbitrary field of zero characteristic. Let us assume that $\delta$ is a Lie bialgebra
structure on $sl(n,\mathbb{F})$. Then one can construct the corresponding classical double $D(sl(n,\mathbb{F}),\delta)$.

Similarly to Lemma 2.1 from \cite{MSZ}, one can prove that
$D(sl(n,\mathbb{F}),\delta)$ is a direct sum of regular adjoint $sl(n)$-modules.
Combining this result with Prop. 2.2 from \cite{BZ}, it follows that
\begin{thm}\label{double}
There exists an associative, unital, commutative algebra $A$, of dimension 2 over $\mathbb{F}$, such that
$D(sl(n,\mathbb{F}),\delta)\cong sl(n,\mathbb{F})\otimes_{\mathbb{F}} A$.
\end{thm}

The symmetric invariant nondegenerate bilinear form $Q$ on $sl(n,\mathbb{F})\otimes_{\mathbb{F}} A$ is given in the following way. For arbitrary elements $f_{1},f_{2}\in sl(n,\mathbb{F})$ and $a,b\in A$ we have
$Q(f_{1}\otimes a,f_{2}\otimes b)=K(f_{1},f_{2})\cdot t(ab)$, where
$K$ denotes the Killing form on $sl(n,\mathbb{F})$ and $t: A\longrightarrow \mathbb{F}$ is a trace function.

Let us now investigate the algebra $A$. Since $A$ is unital
and of dimension 2 over $\mathbb{F}$, one can choose a basis $\{e,1\}$, where $1$ denotes the unit. Moreover, there exist $p$ and $q$ in $\mathbb{F}$ such that
$e^2+pe+q=0$. Let $\Delta=p^2-4q \in\mathbb {F}$. We distinguish the following cases:

(i) Assume $\Delta=0$. Let $\displaystyle \varepsilon:=e+\frac{p}{2}$.
Then $\varepsilon^2=0$
and $A=\mathbb{F}\varepsilon\oplus\mathbb{F}=\mathbb{F}[\varepsilon]$.

(ii) Assume $\Delta$ is the square of a nonzero element of $\mathbb{F}$. In this case, one can choose $e'\in \mathbb{F}^{*}$ such that $e'^{2}=\Delta$. Then $A=\mathbb{F}\oplus e'\mathbb{F}=\mathbb{F}\oplus \mathbb{F}$.

(iii) Assume $\Delta$ is not a square of an element of $\mathbb{F}$. Then $A=\mathbb{F}+e'\mathbb{F}$, where $e'=e+p/2$ and $e'^{2}=\Delta/4\in \mathbb{F}$. Thus $A$ is a quadratic field extension of $\mathbb{F}$.

Summing up the above observations, we get

\begin{thm}
Let $\delta$ be an arbitrary Lie bialgebra structure on $sl(n,\mathbb{F})$. Then $D(sl(n,\mathbb{F}), \delta)\cong sl(n,\mathbb{F})\otimes_{\mathbb{F}} A$, where  $A=\mathbb{F}[\varepsilon]$ and $\varepsilon^{2}=0$, $A=\mathbb{F}\oplus\mathbb{F}$ or $A$ is a quadratic field extension of $\mathbb{F}$.
\end{thm}

The classification of Lie bialgebras with classical double $sl(n,\mathbb{F}[\varepsilon])$ leads to the classification of quasi-Frobenius Lie algebras over $\mathbb{F}$. More precisely, due to the correspondence between Lie bialgebras and Manin triples (see \cite{D}), the following result holds:

\begin{prop}
There exists a one-to-one correspondence between Lie bialgebra structures on
$sl(n,\mathbb{F})$ whose corresponding double is isomorphic to
$sl(n,\mathbb{F}[\varepsilon])$ and Lagrangian subalgebras $W$ of
$sl(n,\mathbb{F}[\varepsilon])$ complementary to $sl(n,\mathbb{F})$.

\end{prop}

Similarly to Theorem 3.2 from \cite{S2}, one can prove
\begin{prop}
Any  Lagrangian subalgebra $W$ of
$sl(n,\mathbb{F}[\varepsilon])$ complementary to $sl(n,\mathbb{F})$ is uniquely defined by a subalgebra $L$ of $sl(n,\mathbb{F})$ together with a nondegenerate 2-cocycle $B$ on $L$.
\end{prop}

We recall that a Lie algebra is called quasi-Frobenius
if there exists a nondegenerate 2-cocycle on it. The complete classification
of quasi-Frobenius Lie subalgebras of $sl(n,\mathbb{F})$ is not generally known
for large $n$.

\section{Belavin--Drinfeld untwisted cohomologies}

Unlike the previous case, the classification of Lie bialgebras with classical double
$sl(n,\mathbb{F})\oplus sl(n,\mathbb{F})$ can be achieved by cohomological and combinatorial methods.

\begin{lem}
Any Lie bialgebra structure $\delta$ on $sl(n,\mathbb{F})$
for which the associated classical double is isomorphic to $sl(n,\mathbb{F})\oplus sl(n,\mathbb{F})$ is a coboundary $\delta=dr$ given by an $r$-matrix satisfying $r+r^{21}=f\Omega$, where $f\in\mathbb{F}$ and $\mathrm{CYB}(r)=0$.
\end{lem}

We may suppose that $f=1$. Naturally we want to classify all such $r$ up to $GL(n,\mathbb{F})$-equivalence. Let $\overline{\mathbb{F}}$ denote the algebraic closure of $\mathbb{F}$. Any Lie bialgebra structure $\delta$ over $\mathbb{F}$ can be extended to a Lie bialgebra structure $\overline{\delta}$ over $\overline{\mathbb{F}}$.

According to \cite{BD}, Lie bialgebra structures on a simple Lie algebra
$\mathfrak{g}$ over
an algebraically closed field are coboundaries given by non-skewsymmetric $r$-matrices.
Suppose we have fixed a Cartan subalgebra $\mathfrak{h}$ and the corresponding root system. Any $r$-matrix depends on a discrete and a continuous parameter. The discrete parameter is an admissible triple $(\Gamma_{1},\Gamma_{2},\tau)$, i.e.
an isometry $\tau:\Gamma_{1}\longrightarrow \Gamma_{2}$ where $\Gamma_{1},\Gamma_{2}\subset\Gamma$ such that
for any $\alpha\in\Gamma_{1}$ there exists $k\in \mathbb{N}$ satisfying
$\tau^{k}(\alpha)\notin \Gamma_{1}$. The continuous parameter is a tensor $r_{0}\in \mathfrak{h}\otimes \mathfrak{h}$ satisfying $r_{0}+r_{0}^{21}=\Omega_{0}$
and $(\tau(\alpha)\otimes 1+1 \otimes \alpha)(r_{0})=0$ for any $\alpha\in \Gamma_{1}$. Here $\Omega_{0}$ denotes the Cartan part of the quadratic Casimir element $\Omega$.
Then the associated $r$-matrix is given by the following formula
\[r=r_{0}+\sum_{\alpha>0}e_{\alpha}\otimes e_{-\alpha}+\sum_{\alpha\in (Span \Gamma_{1})^{+} }\sum_{k\in \mathbb{N}} e_{\alpha}\wedge e_{-\tau^{k}(\alpha)}.\]

Now, let us assume that $\delta$ is a Lie bialgebra structure on $sl(n,\mathbb{F})$. Then its extension $\overline{\delta}$ has a corresponding $r$-matrix. Up to  $GL(n,\overline{\mathbb{F}})$-equivalence, we have the Belavin--Drinfeld classification. We may therefore assume that our $r$-matrix is of the form $r_{X}=(\mathrm{Ad}_{X}\otimes \mathrm{Ad}_{X})(r)$, where $X\in GL(n,\overline{\mathbb{F}})$  and $r$ satisfies the system  $r+r^{21}=\Omega$ and $\mathrm{CYB}(r)=0$.

Let $\sigma\in Gal(\overline{\mathbb{F}}/\mathbb{F})$. Since $\delta(a)=[r_{X},a\otimes1+1\otimes a]$ for any $a\in sl(n,\mathbb{F})$, we have $(\sigma\otimes \sigma)(\delta(a))=[\sigma(r_{X}),a\otimes1+1\otimes a]$ and $(\sigma\otimes \sigma)(\delta(a))=\delta(a)$. Consequently, $\sigma(r_{X})=r_{X}+\lambda\Omega$, for some $\lambda\in\overline{\mathbb{F}}$. Let us show that $\lambda=0$. Indeed, $\Omega=\sigma(\Omega)=\sigma(r_{X})+\sigma(r_{X}^{21})= r_{X}+r_{X}^{21}+2\lambda\Omega$. Thus $\lambda=0$ and $\sigma(r_{X})=r_{X}$.
Consequently, $(\mathrm{Ad}_{X^{-1}\sigma(X)}\otimes \mathrm{Ad}_{X^{-1}\sigma(X)})(\sigma(r))=r$. We recall the following
\begin{defn}
Let $r$ be an $r$-matrix. The \emph{centralizer} $C(r)$ of $r$
is the set of all
$X\in GL(n,\overline{\mathbb{F}})$
satisfying $(\mathrm{Ad}_{X}\otimes \mathrm{Ad}_{X})(r)=r$.
\end{defn}

Using the same argumets as in the proof of Theorem 4.3 \cite{SP1}, it follows that
$\sigma(r)=r$ and $X^{-1}\sigma(X)\in C(r)$, for any $\sigma\in Gal(\overline{\mathbb{F}}/\mathbb{F})$.

\begin{defn}
Let $r$ be a non-skewsymmetric $r$-matrix from the Belavin--Drinfeld list
and $C(r)$ its centralizer. We say that $X\in GL(n,\overline{\mathbb{F}})$ is a \emph{Belavin--Drinfeld cocycle}
associated to $r$ if $X^{-1}\sigma(X)\in C(r)$, for any $\sigma\in Gal(\overline{\mathbb{F}}/\mathbb{F})$.

\end{defn}
The set of Belavin--Drinfeld cocycles associated to $r$ will be denoted by
$Z(sl(n,\mathbb{F}),r)$. Note that this set contains the identity.

\begin{defn}
 Two cocycles $X_1$ and $X_{2}$ in $Z(sl(n,\mathbb{F},r)$ are called \emph{equivalent} if
there exists $Q\in GL(n,\mathbb{F})$ and $C\in C(r)$ such that $X_{1}=QX_{2}C$.
\end{defn}

\begin{defn}
Let $H_{BD}^{1}(sl(n,\mathbb{F}),r)$ denote the set of equivalence classes of cocycles from $Z(sl(n,\mathbb{F}),r)$. We call this set the \emph{Belavin--Drinfeld cohomology} associated to the $r$-matrix $r$. The Belavin--Drinfeld cohomology is said to be \emph{trivial} if all cocycles are equivalent to the identity, and
\emph{non-trivial} otherwise.
\end{defn}

Combining the above definitions with the preceding discussion, we obtain
\begin{prop}
For any non-skewsymmetric $r$-matrix $r$, there exists a one-to-one correspondence between $H^{1}_{BD}(sl(n,\mathbb{F}),r)$
and gauge equivalence classes of Lie bialgebra structures on $sl(n,\mathbb{F})$ with classical double isomorphic to $sl(n,\mathbb{F})\oplus sl(n,\mathbb{F})$ and $\overline{\mathbb{F}}$--isomorphic to $\delta=dr$.

\end{prop}

The Belavin--Drinfeld cohomology set can be computed as in \cite{SP1} and the following result holds.

\begin{thm} For any non-skewsymmetric $r$-matrix $r$, $H^{1}_{BD}(sl(n,\mathbb{F}),r)$ is trivial. Any Lie bialgebra structure on $sl(n,\mathbb{F})$
with classical double $sl(n,\mathbb{F})\oplus sl(n,\mathbb{F})$ is of the form $\delta=dr$, where $r$ is an $r$-matrix which is, up to a multiple from $\mathbb{F}^{*}$, $GL(n,\mathbb{F})$--equivalent to a non-skewsymmertric $r$-matrix from the Belavin--Drinfeld list.
\end{thm}

\section{Belavin--Drinfeld twisted cohomologies}

We focus on the study of Lie bialgebra structures on $sl(n,\mathbb{F})$ whose classical double is
isomorphic to $sl(n,\mathbb{F})\otimes _{\mathbb{F}} A$,
where $A$ is a quadratic extension of $\mathbb{F}$. We may suppose that
$A=\mathbb{F}(\sqrt{d})$, where $d$ is not a square in $\mathbb{F}$.
We will show that Lie bialgebras of this type can also be classified by means of certain cohomology sets.

Twisted cohomologies associated to $r$-matrices for $sl(n,\mathbb{F})$
can be defined as in \cite{SP1}, where we studied the particular case
$\mathbb{F}=\mathbb{C}((\hbar))$. First, similarly to Prop. 5.3 of \cite{SP1}, one can prove the following

\begin{prop}

Any Lie bialgebra structure on $sl(n,\mathbb{F})$ with classical double isomorphic to $sl(n,\mathbb{F}[\sqrt{d}])$ is given by an $r$-matrix $r'$
which satisfies $CYB(r')=0$ and $r'+r'_{21}=\sqrt{d}\Omega$.

\end{prop}

Over $\overline{\mathbb{F}}$, all $r$-matrices are gauge equivalent to the ones from Belavin--Drinfeld list.
It follows that there exists a non-skewsymmetric $r$-matrix $r$ and $X\in GL(n,\overline{\mathbb{F}})$
such that $r'=\sqrt{d}(\mathrm{Ad}_{X}\otimes \mathrm{Ad}_{X})(r)$.

The field $\mathbb{F}[\sqrt{d}]$ is endowed with a conjugation $\overline{a+b\sqrt{d}}=a-b\sqrt{d}$. Denote by $\sigma_{2}$ its lift to $Gal(\overline{\mathbb{F}}/\mathbb{F})$. If $X\in GL(n,\mathbb{F}[\sqrt{d}])$, then $\sigma_{2}(X)=\overline{X}$. Now let us consider the action of $\sigma_{2}$ on $r'$. We have
$\sigma_{2}(r')=r'+\lambda\Omega$, for some $\lambda\in \overline{\mathbb{F}}$. Let us show that $\lambda=-\sqrt{d}$. Indeed, since $r'+r'_{21}=\sqrt{d}\Omega$, we also have $\sigma_{2}(r')+\sigma_{2}(r'_{21})=-\sqrt{d}\Omega$. Combining these relations with $\sigma_{2}(r')=r'+\lambda\Omega$, we get $\lambda=-\sqrt{d}$ and therefore $\sigma_{2}(r')=r'-\sqrt{d}\Omega=-r'_{21}$.

Recall now that $r'=\sqrt{d}(\mathrm{Ad}_{X}\otimes \mathrm{Ad}_{X})(r)$. Then
condition $\sigma_{2}(r')=-r'_{21}$ implies $(\mathrm{Ad}_{X^{-1}\sigma_{2}(X)}\otimes \mathrm{Ad}_{X^{-1}\sigma_{2}(X)})(\sigma_{2}(r))=r^{21}$.

For any $\sigma \in Gal(\overline{\mathbb{F}}/\mathbb{F}[\sqrt{d}])$, $\sigma(r')=r'$, which in turn implies $(\mathrm{Ad}_{X^{-1}\sigma(X)}\otimes \mathrm{Ad}_{X^{-1}\sigma(X)})(\sigma(r))=r$. Now, using the same type of arguments as in the proof of Theorem 4.3 \cite{SP1}, one can deduce
that $\sigma(r)=r$ and therefore the following result holds.

\begin{prop}
Any Lie bialgebra structure on $sl(n,\mathbb{F})$ with
classical double isomorphic to $sl(n,\mathbb{F}[\sqrt{d}])$ is given by
$r'=\sqrt{d}(\mathrm{Ad}_{X}\otimes \mathrm{Ad}_{X})(r)$,
where $r$ is, up to a multiple
from $\mathbb{F}^{*}$, a non-skewsymmetric
 $r$-matrix from the Belavin--Drinfeld list and $X\in GL(n,\overline{\mathbb{F}})$ satisfies $(\mathrm{Ad}_{X^{-1}\sigma_{2}(X)}\otimes \mathrm{Ad}_{X^{-1}\sigma_{2}(X)})(r)=r^{21}$ and $(\mathrm{Ad}_{X^{-1}\sigma(X)}\otimes \mathrm{Ad}_{X^{-1}\sigma(X)})(r)=r$, for any $\sigma \in Gal(\overline{\mathbb{F}}/\mathbb{F}[\sqrt{d}])$.

\end{prop}

\begin{defn}
Let $r$ be a non-skewsymmetric $r$-matrix from the Belavin--Drinfeld list.
We say that $X\in GL(n,\overline{\mathbb{F}})$ is a \emph{Belavin--Drinfeld twisted cocycle} associated to $r$ if $(\mathrm{Ad}_{X^{-1}\sigma_{2}(X)}\otimes \mathrm{Ad}_{X^{-1}\sigma_{2}(X)})(r)=r^{21}$ and $(\mathrm{Ad}_{X^{-1}\sigma(X)}\otimes \mathrm{Ad}_{X^{-1}\sigma(X)})(r)=r$, for any $\sigma \in Gal(\overline{\mathbb{F}}/\mathbb{F}[\sqrt{d}])$.
\end{defn}

The set of Belavin--Drinfeld twisted cocycle associated to $r$
will be denoted by $\overline{Z}(sl(n,\mathbb{F}),r)$. Let us analyse
for which admissible triples this set is non-empty.

Let $S\in GL(n,\mathbb{F})$ be the matrix with 1 on the second diagonal and 0 elsewhere. Let us denote by $s$ the automorphism of the Dynkin diagram given by $s(\alpha_{i})=\alpha_{n-i}$ for all $i\leq n-1$.

\begin{prop}\label{cond_tau}
Let $r$ be a non-skewsymmetric $r$-matrix associated to an admissible triple
$(\Gamma_{1},\Gamma_{2},\tau)$. If $\overline{Z}(sl(n,\mathbb {F})),r)\neq \emptyset$, then $s(\Gamma_{1})=\Gamma_{2}$ and $s\tau=\tau^{-1}s$.

\end{prop}

\begin{defn}
Let $X_{1}$ and $X_{2}$ be two Belavin--Drinfeld twisted cocycles associated to
 $r$. We say that they are \emph{equivalent} if
there exists $Q\in GL(n,\mathbb{F})$ and $C\in C(r)$ such that $X_{2}=QX_{1}C$.
\end{defn}

The set of equivalence classes of twisted cocycles corresponding to
a non-skewsymmetric $r$-matrix $r$ will be denoted by $\overline{H}^{1}_{BD}(sl(n,\mathbb{F}),r)$.

\begin{rem} If two twisted cocycles $X_{1}$ and $X_{2}$ are equivalent,
then the corresponding
$r$-matrices $\sqrt{d}(\mathrm{Ad}_{X_{1}}\otimes \mathrm{Ad}_{X_{1}})(r)$ and
$\sqrt{d} (\mathrm{Ad}_{X_{2}}\otimes \mathrm{Ad}_{X_{2}})(r)$ are gauge equivalent via $Q$.

\end{rem}
\begin{rem}\label{rem}
In fact, by obvious reasons it is better to denote  $\overline{H}^{1}_{BD}(sl(n,\mathbb{F}),r)$
by  $\overline{H}^{1}_{BD}(sl(n,\mathbb{F}),r,d)$. However,  in this section we fix $d$
and the notation $\overline{H}^{1}_{BD}(sl(n,\mathbb{F}),r)$ is not misleading. We will use
$\overline{H}^{1}_{BD}(sl(n,\mathbb{F}),r,d)$ in the next section.
\end{rem}
 \begin{prop}\label{co}
There exists a one-to-one correspondence between the twisted cohomology set
$\overline{H}_{BD}^{1}(sl(n,\mathbb{F}),r)$ and
gauge equivalence classes of Lie bialgebra structures on $sl(n,\mathbb{F})$
with classical double isomorphic to $sl(n,\mathbb{F}[\sqrt{d}])$
and $\overline{\mathbb{F}}$--isomorphic to $\delta=dr$.
\end{prop}

Let $r_{DJ}$ be the Drinfeld--Jimbo $r$-matrix. Having fixed a Cartan subalgebra $\mathfrak{h}$ of $\mathfrak{g}$ and the associated root system, we choose a system of generators $e_{\alpha}$, $e_{-\alpha}$,
$h_{\alpha}$ where $\alpha>0$ such that $K(e_{\alpha},e_{-\alpha})=1$. Denote by $\Omega_{0}$ the Cartan part of $\Omega$. Then $$r_{DJ}=\sum_{\alpha>0}e_{\alpha}\otimes e_{-\alpha}+\frac{1}{2}\Omega_{0}.$$

The twisted cohomology corresponding to $r_{DJ}$ can be studied in the same manner as was done in
\cite{SP1} (se Prop. 7.15). Let $J\in GL(n,\mathbb{F}[\sqrt{d}])$ denote the matrix with entries
$a_{kk}=1$ for $k \leq m$, $a_{kk}=-\sqrt{d}$ for $k \geq m+1$, $a_{k,n+1-k}=1$
for $k \leq m$, $a_{k,n+1-k}=\sqrt{d}$ for $k \geq m+1$, where $m=[\frac{n+1}{2}]$.

\begin{thm}\label{case DJ}
The Belavin--Drinfeld twisted cohomology
$\overline{H}_{BD}^{1}(sl(n),r_{DJ})$ is non-empty and consists of one element, the class of $J$.
\end{thm}

\begin{proof}

Let $X$ be a twisted cocycle associated to $r_{DJ}$. Then $X$ is equivalent to a twisted cocycle $P\in GL(n,\mathbb{F}[\sqrt{d}])$,
associated to $r_{DJ}$. We may therefore assume from the beginning that
$X\in GL(n,\mathbb{F}[\sqrt{d}])$ and it remains to prove that all such cocycles
are equivalent. The proof will be done by induction.

For $n=2$, consider $J=\left(\begin{array}{cc} 1 & 1 \\
\sqrt{d} &-\sqrt{d} \end{array}\right)$ and let $X=(a_{ij})\in GL(2,\mathbb{F}[\sqrt{d}])$
satisfy $\overline{X}=XSD$ with $D=\mathrm{diag}(d_{1},d_{2})\in GL(2,\mathbb{F}
[\sqrt{d}])$. The identity is equivalent to the following system: $\overline{a_{11}}=a_{12}d_{1}$,  $\overline{a_{12}}=a_{11}d_{2}$, $\overline{a_{21}}=a_{22}d_{1}$,
$\overline{a_{22}}=a_{21}d_{2}$. Assume that $a_{21}a_{22}\neq 0$. Let $a_{11}/a_{21}=a'+b'\sqrt{d}$. Then $a_{12}/a_{22}=a'-b'\sqrt{d}$. One can immediately check that $X=QJD'$, where $Q=\left(\begin{array}{cc} a' & b' \\
1 &0 \end{array}\right)\in GL(2,\mathbb{F})$, $D'=\mathrm{diag}(a_{21},a_{22})\in \mathrm{diag}(2,\mathbb{F}[\sqrt{d}])$.

For $n=3$, set
$J=\left(\begin{array}{ccc} 1 & 0 &1\\
0&1&0\\\sqrt{d}&0&-\sqrt{d}\end{array}\right)$ and let $X=(a_{ij})\in GL(3,\mathbb{F}[\sqrt{d}])$ satisfy $\overline{X}=XSD$,
with $D=\mathrm{diag}(d_{1},d_{2},d_{3})\in GL(3,\mathbb{K}[\sqrt{d}])$. The identity is equivalent to the following system: $\overline{a_{11}}=d_{1}a_{13}$, $\overline{a_{21}}=d_{1}a_{23}$, $\overline{a_{31}}=d_{1}a_{33}$, $\overline{a_{12}}=d_{2}a_{12}$, $\overline{a_{22}}=d_{2}a_{22}$, $\overline{a_{32}}=d_{2}a_{32}$, $\overline{a_{13}}=d_{3}a_{11}$, $\overline{a_{23}}=d_{3}a_{21}$, $\overline{a_{33}}=d_{3}a_{31}$. Assume that $a_{21}a_{22}a_{23}\neq 0$.

Let  $a_{11}/a_{21}=b_{11}+b_{13}\sqrt{d}$ and $a_{31}/a_{21}=b_{31}+b_{33}\sqrt{d}$. Then $a_{13}/a_{23}=b_{11}-b_{13}\sqrt{d}$ and $a_{33}/a_{23}=b_{31}-b_{33}\sqrt{d}$. On the other hand, let  $b_{12}:=a_{12}/a_{22}$ and $b_{32}:=a_{32}/a_{22}$. Note that $b_{12}\in \mathbb{F}$, $b_{32}\in \mathbb{F}$. One can immediately check that $X=QJD'$, where
$Q=\left(\begin{array}{ccc} b_{11} & b_{12} &b_{13}\\
1&1&0\\ b_{31} & b_{32} & b_{33}\end{array}\right)\in GL(3,\mathbb{F})$, $D'=\mathrm{diag}(a_{21},a_{22},a_{23})\in \mathrm{diag}(3,\mathbb{F}[\sqrt{d}])$.

Assume $n>3$. Let us denote the constructed above $J\in GL(n,\mathbb{F}[\sqrt{d}])$ by $J_n$. We are going to prove that if $X\in GL(n,\mathbb{F}[\sqrt{d}])$ satisfies $\overline{X}=XSD$, then using elementary row operations with entries from $\mathbb{F}$ and multiplying columns by proper elements from
$\mathbb{F}[\sqrt{d}]$ we can bring $X$ to $J_n$.

We will need the following operations on a matrix
$M=\{m_{pq}\}\in \mathrm{Mat}(n)$

1. $u_n (M)=\{m_{pq},\ p,q=2,3,...,n-1\}\in \mathrm{Mat}(n-2)$;

2. $g_n (M)=\{m_{pq}\}\in \mathrm{Mat}(n+2),$ where $m_{pq}$ are already defined for $p,q=1,2,...n$,
$m_{00} =m_{n+1,n+1}=1$ and the rest $m_{0,a}=m_{a,0}=m_{n+1,a}=m_{a,n+1}=0$.
\vskip0.2cm
It is clear, that $u_n (X)$ is a cocycle in $GL(n-2, \mathbb{F}[\sqrt{d}])$
and by induction, there exist $Q_{n-2}\in GL(n-2, \mathbb{F})$ and a diagonal matrix $D_{n-2}$ such that
$$
Q_{n-2}\cdot u_n (X)\cdot D_{n-2} =J_{n-2}
$$
Let us consider $X_n=g_{n-2} (Q_{n-2} )\cdot X\cdot g_{n-2} (D_{n-2})$. Clearly, $X_n$ is a twisted cocycle
equivalent to $X$ and $u_n (X_n )=J_{n-2}$.

Applying elementary row operations with entries from $\mathbb{F}$ and multiplying by a proper
diagonal matrix we can obtain a new cocycle
$Y_n=(y_{pg})$ equivalent to $X$ with the following properties:

1. $u_n (Y_n)=J_{n-2}$;

2. $y_{12}=y_{13}=...=y_{1,n-1}=0$ and $y_{n2}=y_{n3}=...=y_{n,n-1}=0$;

3. $y_{11}=y_{1n}=1$, here we use the fact that if $y_{pq}=0$, then $y_{p,n+1-q}=0$.

It follows from the cocycle condition $\overline{Y_n} =Y_n \cdot S\cdot\mathrm{diag}(h_1,...,h_n)$
that $h_1=h_n=1$ and hence, $y_{n1}=\overline{y_{nn}}$.

Now, we can use the first row to achieve  $y_{n1}=-y_{nn}=\sqrt{d}$ and after that, we use the first
and the last rows to ``kill'' $\{y_{k1},\ k=2,...,n-1\}$. Then   the set $\{y_{kn},\ k=2,...,n-1\}$ will be ``killed''
automatically.  We have obtained $J_n$ from $X$ and thus, have proved that $X$ is equivalent to $J_n$.

\end{proof}
Let us now investigate twisted cohomologies associated to arbitrary non-skewsymmetric $r$-matrices.
The following two results will prove to be useful for our study.

\begin{lem}\label{equiv}

Assume $X\in \overline{Z}(sl(n),r)$. Then there exists a twisted cocycle $Y\in GL(n,\mathbb{F}[\sqrt{d}])$, associated to $r$, and equivalent to $X$.

\end{lem}

\begin{proof}
We have $X\in GL(n,\overline{\mathbb{F}})$ and for any $\sigma\in Gal(\overline{\mathbb{F}}/\mathbb{F}[\sqrt{d}])$,
$X^{-1}\sigma(X)\in C(r)$. On the other hand, the Belavin--Drinfeld cohomology for $sl(n)$ associated to $r$ is trivial. This implies that $X$ is equivalent to the identity, where in the equivalence relation
we consider $\mathbb{F}[\sqrt{d}]$
instead of $\mathbb{F}$. So there exists $Y\in GL(n,\mathbb{F}[\sqrt{d}])$ and $C\in C(r)$ such that $X=YC$. On the other hand, $Y\in \overline{Z}(sl(n),r)$
 since $(\mathrm{Ad}_{X^{-1}\sigma_{2}(X)}\otimes \mathrm{Ad}_{X^{-1}\sigma_{2}(X)})(r)=r^{21}$ implies $(\mathrm{Ad}_{Y^{-1}\sigma_{2}(Y)}\otimes \mathrm{Ad}_{Y^{-1}\sigma_{2}(Y)})(r)=r^{21}$.
\end{proof}

\begin{prop}\label{decRJD}
Let $r$ be a non-skewsymmetric $r$-matrix
associated to an admissible triple $(\Gamma_{1},\Gamma_{2},\tau)$
satisfying $s(\Gamma_{1})=\Gamma_{2}$ and $s\tau=\tau^{-1}s$. If
$X\in \overline{Z}(sl(n,\mathbb{F})), r)$, then there exist $R\in GL(n,\mathbb{F})$ and $D\in \mathrm{diag}(n, \overline{\mathbb{F}})$ such that $X=RJD$.

\end{prop}

\begin{proof}
According to Lemma \ref{equiv}, $X=YC$, where $Y\in GL(n, \mathbb{F}[\sqrt{d}])$
and $C\in C(r)$. Since $(\mathrm{Ad}_{Y^{-1}\sigma_{2}(Y)}\otimes \mathrm{Ad}_{Y^{-1}\sigma_{2}(Y)})(r)=r^{21}$ and $(\mathrm{Ad}_{S}\otimes \mathrm{Ad}_{S})(r)=r^{21}$,
it follows that $S^{-1}Y^{-1}\sigma_{2}(Y)\in C(r)$. On the other hand,
by Lemma 4.11 from \cite{SP1}, $C(r)\subset \mathrm{diag}(n,\overline{\mathbb{F}})$. We get
$S^{-1}Y^{-1}\sigma_{2}(Y)\in \mathrm{diag}(n,\overline{\mathbb{F}})$. Now Theorem \ref{case DJ} implies that $Y=RJD_{0}$, where $R\in GL(n,\mathbb{F})$ and $D_{0}\in \mathrm{diag}(n,\overline{\mathbb{F}})$. Consequently, $X=RJD_{0}C=RJD$ with
$D=D_{0}C\in \mathrm{diag}(n,\overline{\mathbb{F}})$.

\end{proof}

Let $T$ denote the automorphism of $\mathrm{diag}(n, \overline{\mathbb{F}})$
defined by $T(D)=SD^{-1}S\overline{D}$.

\begin{lem}\label{lem1_T}
Let $r$ be a non-skewsymmetric $r$-matrix with centralizer $C(r)$. Let $X=RJD$, with $R\in GL(n,\mathbb{F})$ and $D\in \mathrm{diag}(n, \mathbb{F}[\sqrt{d}])$. Then $X \in \overline{Z}(sl(n,\mathbb{F}),r)$ if and only if $D\in T^{-1}(C(r))$.

\end{lem}

\begin{proof}
Let us first note that $X\in \overline{Z}(sl(n,\mathbb{F}),r)$ if and only if for any $\sigma\in Gal(\overline{\mathbb{F}}/\mathbb{F}[\sqrt{d}])$,
$X^{-1}\sigma(X)\in C(r)$ and $SX^{-1}\overline{X}\in C(r)$.
We have $X=RJD$ which implies $\overline{X}=\overline{R}\overline{J}\overline{D}=RJS\overline{D}=RJDD^{-1}S\overline{D}=XD^{-1}S\overline{D}=XST(D)$.
We immediately get
that $SX^{-1}\overline{X}\in C(r)$ if and only if $T(D)\in C(r)$.

\end{proof}

\begin{lem}\label{lem2_T}
Let $X_{1}=R_{1}JD_{1}$ and $X_{2}=R_{2}JD_{2}$ be two Belavin--Drinfeld
twisted cocycles associated to $r$. Then $X_{1}$ and $X_{2}$ are equivalent
if and only if
$D_{2}D_{1}^{-1} \in C(r)\cdot Ker(T)$.
\end{lem}

\begin{proof}
Assume the two cocycles are equivalent. There exist $Q\in GL(n,\mathbb{F})$ and $C\in C(r)$ such that $X_{2}=QX_{1}C$. Then $Q=R_{2}JD_{2}C^{-1}D_{1}^{-1}J^{-1}R_{1}^{-1}$. Since $Q=\overline{Q}$
and $\overline{J}=JS$, we get
$D_{2}C^{-1}D_{1}^{-1}=S\overline{D_{2}C^{-1}D_{1}^{-1}}S$. Thus $D_{2}C^{-1}D_{1}^{-1}\in Ker(T)$.
On the other hand,
$C\in C(r)\subset \mathrm{diag}(n,\overline{\mathbb{F}})$,
so $D_{2}C^{-1}D_{1}^{-1}=
D_{2}D_{1}^{-1}C^{-1}$. We have obtained that $D_{2}D_{1}^{-1} \in C(r)\cdot Ker(T)$.
Conversely, if this condition is satisfied, then write
$D_{2}D_{1}^{-1}=D_{0}C$, where $C\in C(r)$ and $D_{0}\in Ker(T)$.
Denote $Q:=R_{2}JD_{0}J^{-1}R_{1}^{-1}$. Then, by construction, $Q=\overline{Q}$ and $X_{2}=QX_{1}C$.
\end{proof}

By lemmas \ref{lem1_T} and \ref{lem2_T}, we get

\begin{prop}

Let $r$ be a non-skewsymmetric $r$-matrix
associated to an admissible triple $(\Gamma_{1},\Gamma_{2},\tau)$
satisfying $s(\Gamma_{1})=\Gamma_{2}$ and $s\tau=\tau^{-1}s$.
Then \[\overline{H}_{BD}^{1}(sl(n,\mathbb{F}),r)=\frac{T^{-1}(C(r))}{C(r)\cdot Ker(T)}.\]
\end{prop}

At this point, one needs the explicit description of the centralizer
and its preimage under $T$.

\begin{lem}\label{center_r}

Let  $r$ be a non-skewsymmetric $r$-matrix
associated to an admissible triple $(\Gamma_{1},\Gamma_{2},\tau)$.
Then the following hold:

(a) $C(r)$ consists of all diagonal matrices $D=\mathrm{diag}(d_{1},...,d_{n})$ such that $d_{i}=s_{i}s_{i+1}...s_{n}$, where $s_{i}\in \overline{\mathbb{F}}$ satisfy the condition: $s_{i}=s_{j}$ if  $\alpha_{i}\in \Gamma_{1}$ and $\tau(\alpha_{i})=\alpha_{j}$.

(b) $T^{-1}(C(r))$ consists of all diagonal matrices $D=\mathrm{diag}(d_{1},...,d_{n})$ such that $d_{i}=s_{i}s_{i+1}...s_{n}$, where $s_{i}\in \overline{\mathbb{F}}$ satisfy the condition: $\overline{s}_{i}s_{n-i}=\overline{s}_{j}s_{n-j}$ if  $\alpha_{i}\in \Gamma_{1}$ and $\tau(\alpha_{i})=\alpha_{j}$.

\end{lem}

\begin{proof}
Part (a) can be proved in the same way as Lemma 5.5 from \cite{SP1} and (b) follows immediately from (a).
\end{proof}

Let us make the following remark. Any admissible triple $(\Gamma_{1},\Gamma_{2},\tau)$ can be viewed as a union of strings $\beta_{1}\xrightarrow {\tau}\beta_{2}\xrightarrow{\tau} ...\xrightarrow{\tau}\beta_{k}$, where
$\beta_{1}=\alpha_{i_{1}}\in \Gamma_{1}$,...., $\beta_{k-1}=\alpha_{i_{k-1}}\in \Gamma_{1}$ and $\beta_{k}=\alpha_{i_{k}}\notin \Gamma_{1}$. The above lemma implies that elements of $C(r)$ have the property that $s_{i_{1}}=s_{i_{2}}=...=s_{i_{k}}$, i.e. $s_{i}$ is constant on each string.
In turn, elements of $T^{-1}(C(r))$ satisfy $\overline{s}_{i_{1}}s_{n-i_{1}}=
\overline{s}_{i_{2}}s_{n-i_{2}}=...=\overline{s}_{i_{k}}s_{n-i_{k}}$, i.e. $\overline{s}_{i}s_{n-i}$ is constant on each string.

\begin{thm}\label{main}
Suppose $r$ is a non-skewsymmetric $r$-matrix with admissible triple $(\Gamma_{1},\Gamma_{2},\tau)$ satisfying $s\tau =\tau^{-1}s$. Let $str(\Gamma_{1},\Gamma_{2},\tau)$ denote the number of symmetric strings without middlepoint. Then

\[\overline{H}_{BD}^{1}(sl(n,\mathbb{F}),r)=
\left(\frac{\mathbb{F}^{*}}{N_{\mathbb{F}(\sqrt{d})/\mathbb{F}}(\mathbb{F}(\sqrt{d}))^{*}}\right)^{str(\Gamma_{1},\Gamma_{2},\tau)}.\]

\end{thm}

\begin{proof}
Let $\varphi:(\overline{\mathbb{F}}^{*})^{n}\rightarrow \mathrm{diag}(n,\overline{\mathbb{F}})$ be the map
\[\varphi(s_{1},...,s_{n-1},s_{n})=\mathrm{diag}(s_{1}...s_{n},s_{2}...s_{n},...,s_{n-1}s_{n},s_{n}).\] Consider $\widetilde{T}=\varphi^{-1}T\varphi$. Since $Ker(T)=\varphi Ker(\widetilde{T})$, we have

\[\frac{T^{-1}(C(r))}{Ker(T)\cdot C(r)}\cong \frac{\widetilde{T}^{-1}\varphi^{-1}(C(r))}{Ker(\widetilde{T})\cdot\varphi^{-1}(C(r))}.\] We make the following remarks:

(i) $(s_{1},...,s_{n})\in Ker(\widetilde{T})$ if and only if $\overline{s}_{i}s_{n-i}=1$ for all $i\leq n-1$ and $\overline{s}_{n}=s_{1}...s_{n}$.

(ii) $(s_{1},...,s_{n})\in\varphi^{-1}(C(r))$ is equivalent to $s_{i}$ is constant on each string of the given triple.

(iii) $(s_{1},...,s_{n})\in\widetilde{T}^{-1}\varphi^{-1}(C(r))$ implies that $\overline{s}_{i}s_{n-i}$ is constant on each string.

Step 1. Suppose that the admissible triple is the disjoint union of two symmetric strings $\alpha_{i_{1}}\xrightarrow {\tau}\alpha_{i_{2}}\xrightarrow{\tau} ...\xrightarrow{\tau}\alpha_{i_{k}}$ and $\alpha_{n-i_{k}}\xrightarrow {\tau}\alpha_{n-i_{k-1}}\xrightarrow{\tau} ...\xrightarrow{\tau}\alpha_{n-i_{1}}$. Here we recall that
$\tau$ has the property that $\tau(\alpha_{n-j})=\alpha_{n-i}$ if $\tau(\alpha_{i})=\alpha_{j}$.

Let $(s_{1},...,s_{n})\in\widetilde{T}^{-1}\varphi^{-1}(C(r))$. Then
$\overline{s}_{i_{1}}s_{n-i_{1}}=...=\overline{s}_{i_{k}}s_{n-i_{k}}=:t$
and $\overline{s}_{n-i_{1}}s_{i_{1}}=...=\overline{s}_{n-i_{k}}s_{i_{k}}=\overline{t}$. One can check that $(s_{1},...,s_{n})\in Ker(\widetilde{T})\cdot \varphi^{-1}(C(r))$. Indeed, let us assume first that $n=2m+1$. Then $(s_{1},...,s_{n})$ is the product of the following elements:
$(s_{1},...,s_{m},(\overline{s}_{m})^{-1},...,(\overline{s}_{1})^{-1},
\overline{s_{1}...s_{m}})$ and
$(1,...,1,s_{m+1}\overline{s}_{m},...,s_{n-1}\overline{s}_{1},s_{n}(\overline{s_{1}...s_{m}})^{-1})$. The first factor belongs to $Ker(\widetilde{T})$ and the second is in $\varphi^{-1}(C(r))$ since the
$n-i_{1}$, ..., $n-i_{k}$ coordinates have the constant value $t$.

Suppose that $n=2m$. Consider
$(s_{1},...,s_{m-1}, r_{m},(\overline{s}_{m+1})^{-1},...,(\overline{s}_{1})^{-1},
\overline{s}_{n})$, where $r_{m}=\frac{\overline{s_{1}...s_{m-1}}s_{n}}{s_{1}...s_{m-1}\overline{s}_{n}}$ and $(1,...,1,s_{m}/r_{m},s_{m+1}\overline{s}_{m-1},...,s_{n-1}\overline{s}_{1})$. The first factor is in  $Ker(\widetilde{T})$ since
$r_{m}\overline{r}_{m}=1$ and the second is in  $\varphi^{-1}(C(r))$ since
neither $n-i_{1}$, ..., $n-i_{k}$ can be $m$, and the corresponding coordinates
all equal $t$.

Step 2. Let us assume that the admissible triple consists of a symmetric string
 $\alpha_{i_{1}}\xrightarrow {\tau}\alpha_{i_{2}}\xrightarrow{\tau} ...\xrightarrow{\tau}\alpha_{i_{k}}\xrightarrow{\tau}...\xrightarrow {\tau}\alpha_{n-i_{k}}\xrightarrow {\tau}\alpha_{n-i_{k-1}}\xrightarrow{\tau} ...\xrightarrow{\tau}\alpha_{n-i_{1}}$ not containing
the middlepoint. Let $(s_{1},...,s_{n})\in\widetilde{T}^{-1}\varphi^{-1}(C(r))$. Then $\overline{s}_{i_{1}}s_{n-i_{1}}=...=\overline{s}_{i_{k}}s_{n-i_{k}}=\overline{s}_{n-i_{1}}s_{i_{1}}=...=\overline{s}_{n-i_{k}}s_{i_{k}}=t$.  We note that $t\in \mathbb{F}$
since $t=\overline{t}$.

Case 1. Assume there exists $q\in \mathbb{F}(\sqrt{d})$ such that $t=q\overline{q}$. Then $(s_{1},...,s_{n})\in Ker(\widetilde{T})\cdot\varphi^{-1}(C(r))$. Indeed,
one can make the same construction as in Step 1, except for the
positions $i_{1}$,..., $i_{k}$, $n-i_{1}$,..., $n-i_{k}$ where we consider instead
the decomposition \[(...,s_{i_{l}},...,s_{n-i_{l}},...)=
(...,\frac{s_{i_{l}}}{q},...,\frac{s_{n-i_{l}}}{q},...)\cdot (...,q,...,q,...). \]

Case 2. Assume for any $q\in \mathbb{F}(\sqrt{d})$, $t\neq q\overline{q}$. Then it follows that
$(s_{1},...,s_{n})\notin Ker(\widetilde{T})\cdot\varphi^{-1}(C(r))$. Indeed, let us assume the contrary, i.e. we may write $s_{i}=p_{i}r_{i}$, where
$\overline{p}_{i}p_{n-i}=1$ for all $i\leq n-1$, $\overline{p}_{n}=p_{1}...p_{n}$ and $r_{i_{1}}=...=r_{i_{k}}=r_{n-i_{1}}=...=r_{n-i_{k}}$. It follows that
 $t=\overline{s}_{i_{1}}s_{n-i_{1}}=\overline{r}_{i_{1}}r_{n-i_{1}}=\overline{r}_{i_{1}}r_{i_{1}}$, which is a contradiction.

Step 3. Let us suppose that the admissible triple consists of
a symmetric string
 $\alpha_{i_{1}}\xrightarrow {\tau}\alpha_{i_{2}}\xrightarrow{\tau} ...\xrightarrow{\tau}\alpha_{i_{k}}\xrightarrow{\tau}...\xrightarrow {\tau}\alpha_{n-i_{k}}\xrightarrow {\tau}\alpha_{n-i_{k-1}}\xrightarrow{\tau} ...\xrightarrow{\tau}\alpha_{n-i_{1}}$ containing  the middlepoint. In this case, $\overline{s}_{i_{1}}s_{n-i_{1}}=...=\overline{s}_{i_{k}}s_{n-i_{k}}=\overline{s}_{n-i_{1}}s_{i_{1}}=...=\overline{s}_{n-i_{k}}s_{i_{k}}=t$. Moreover,
$t=s_{m}\overline{s}_{m}$, where $s_{m}$ is the coordinate corresponding to
the middlepoint $\alpha_{m}$. Then again $(s_{1},...,s_{n})\in {Ker(\widetilde{T})\cdot\varphi^{-1}(C(r))}$ since we may proceed as in Step 2, case 1 by taking
$q=s_{m}$.
\end{proof}

\begin{ex}
For $\mathbb{F}=\mathbb{R}$ and $d=-1$, it follows that
given an $r$-matrix $r$ with admissible triple
$(\Gamma_{1},\Gamma_{2},\tau)$, $\overline{H}_{BD}^{1}(sl(n,\mathbb{R}),r)=(\mathbb{Z}_{2})^{str(\Gamma_{1},\Gamma_{2},\tau)}$.
\end{ex}

\begin{ex}

Let $\mathfrak{g}$ be a simple complex finite-dimensional Lie algebra. According to results of Etingof and Kazhdan (see \cite{EK1, EK2}), there exists an equivalence between the category $HA_{0}(\mathbb{C}[[\hbar]])$
of topologically free Hopf algebras cocommutative modulo $\hbar$ and the category $LBA_{0}(\mathbb{C}[[\hbar]])$ of topologically free over $\mathbb{C}[[\hbar]]$ Lie bialgebras with $\delta\equiv 0 (\mathrm{mod}\hbar)$. Due to this equivalence, the classification of quantum groups whose quasi-classical limit is $\mathfrak{g}$ is equivalent to the classification of Lie bialgebra structures on $\mathfrak{g}\otimes_{\mathbb{C}}\mathbb{C}[[\hbar]]$. This in turn reduces to the problem of finding Lie bialgebras on $\mathfrak{g}\otimes_{\mathbb{C}}\mathbb{C}((\hbar))$.

In \cite{SP1, SP2}  the classification of quantum groups with quasi-classical limit $\mathfrak{g}$ was discussed and a theory of Belavin--Drinfeld cohomology associated to any non-skewsymmetric $r$-matrix was introduced.

Let us consider $\mathbb{F}=\mathbb{C}((\hbar))$ and $d=\hbar$. Then
$N(\mathbb{F}(\sqrt{d}))=\mathbb{F}$ and Theorem \ref{main} implies that
$\overline{H}_{BD}^{1}(sl(n,\mathbb{C}((\hbar))),r)$ is trivial (consists of one element)
for any $r$-matrix $r$ satisfying the condition of Proposition \ref{cond_tau} and empty otherwise.
We have thus generalized our previous results \cite{SP2}, where
we proved that twisted cohomologies for $sl(n)$ associated to
generalized Cremmer--Gervais $r$-matrices are trivial.

\end{ex}
This result completes classification of quantum groups which have $sl(n,\mathbb{C})$ as the classical limit.
Summarizing, we have the following picture:

1. According to \cite{EK1, EK2}, there exists an equivalence between the category $HA_{0}(\mathbb{C}[[\hbar]])$
of topologically free Hopf algebras cocommutative modulo $\hbar$ and the category $LBA_{0}(\mathbb{C}[[\hbar]])$
of topologically free over $\mathbb{C}[[\hbar]]$ Lie bialgebras with $\delta\equiv 0 (\mathrm{mod}\hbar)$.

2. To describe the category  $LBA_{0}(\mathbb{C}[[\hbar]])$, it is sufficient (multiplying by a proper
power of $\hbar^N$) to classify Lie bialgebra structures
on the Lie algebra $\mathfrak{g}\otimes_{\mathbb{C}}\mathbb{C}((\hbar))$.

3. Following \cite{MSZ}, only three classical Drinfeld doubles are possible, namely
$$
D(\mathfrak{g}\otimes_{\mathbb{C}}\mathbb{C}((\hbar)))=\mathfrak{g}\otimes_{\mathbb{C}} A_k ,\ k=1,2,3.
$$
Here $A_1=\mathbb{K}[\varepsilon], \ \varepsilon^2=0$;
$A_2=\mathbb{K}\oplus\mathbb{K}$; $A_3=\mathbb{K}(\sqrt{\hbar})$
with $\mathbb{K}=\mathbb{C}((\hbar))$.

4. Lie bialgebra structures related to the case $A_1$ are in a one-to-one correspondence
with quasi-Frobenius subalgebras of $\mathfrak{g}\otimes_{\mathbb{C}}\mathbb{C}((\hbar))$.

5. Now we turn to the case $D(\mathfrak{g}\otimes_{\mathbb{C}}\mathbb{C}((\hbar)))=\mathfrak{g}\otimes_{\mathbb{C}} A_2$
with $\mathfrak{g}=sl(n)$. Up to multiplication by $\hbar^N$ and conjugation by an element of $GL(n,\mathbb{K})$, the related Lie bialgebra
structures are defined by the Belavin-Drinfeld data (see \cite{BD} and Section 2, the main ingredient is the
triple $\tau:\Gamma_1\to\Gamma_2$) and an additional data called a
Belavin-Drinfeld cohomology. In the case $\mathfrak{g}=sl(n)$, the cohomology consists of {\it one element}
independently of the Belavin-Drinfeld data. As a representative of this cohomology class one can choose the
{\it identity matrix}.

6. Finally, in the case $A_3$ and  $\mathfrak{g}=sl(n)$ the description is as follows.
Up to multiplication by $\hbar^N$ and conjugation by an element of $GL(n,\mathbb{K})$, the related Lie bialgebra
structures are defined by the Belavin-Drinfeld data and an additional data called a {\it twisted}
Belavin-Drinfeld cohomology. In this case the twisted cohomology consists of {\it one element}
if $\tau:\Gamma_1\to\Gamma_2$ satisfies the condition of Proposition \ref{cond_tau} and is {\it empty} otherwise
(no Lie bialgebra structures of the type $A_3$ if $\tau$ does not satisfy the condition of Proposition \ref{cond_tau}).
If the cohomology class is non-empty, it can be represented by the matrix $J$ introduced before Theorem \ref{case DJ}.

\section{Relations between  $\overline{H}_{BD}^{1}(sl(n,\mathbb{F}),r)$ and $\mathrm{Br} (\mathbb{F})$}

In this section we return to notation $\overline{H}_{BD}^{1}(sl(n,\mathbb{F}),r,d)$, see \ref{rem}.
According to Proposition \ref{co},
there exists a one-to-one correspondence between the twisted cohomology set
$\overline{H}_{BD}^{1}(sl(n,\mathbb{F}),r,d)$ and
gauge equivalence classes of Lie bialgebra structures on $sl(n,\mathbb{F})$
with the classical double isomorphic to $sl(n,\mathbb{F}[\sqrt{d}])$.

\begin{defn}    Let us define $\overline{H}_{BD-total}^{1} (sl(n,\mathbb{F}),r)$ as the union  of the sets
$\overline{H}_{BD}^{1}(sl(n,\mathbb{F}),r,d)$, where $d$ runs over $\{\mathbb{F}^{*}/(\mathbb{F}^{*})^2\} {\backslash}\{1\}$.
\end{defn}

Let $r_{CG}\in sl(3,\mathbb{F})^{\otimes 2}$ be the $r$-matrix defined by the triple
$$
\Gamma_1 =\{\alpha\},\ \Gamma_2 =\{\beta\},\ \tau (\alpha)=\beta
$$
If $\mathbb{F}=\mathbb{C}$, it is called the Cremmer--Gervais $r$-matrix. As a corollary to Theorem
\ref{main}, we obtain the following
\begin{thm}\label{br}
$$\overline{H}_{BD-total}^{1} (sl(3,\mathbb{F}),r_{CG})=$$
$$\{ (d,b): d\in\{\mathbb{F}^{*}/(\mathbb{F}^{*})^2\} \setminus \{1\}, b\in\frac{\mathbb{F}^{*}}{N_{\mathbb{F}(\sqrt{d})/\mathbb{F}}\ (\mathbb{F}(\sqrt{d}))^{*}}\}.$$

\end{thm}
Let $\mathrm{Br}(\mathbb{F})$ be the Brauer group of the field $\mathbb{F}$,
which is an abelian group whose elements are Morita equivalence classes of central simple algebras of finite rank over $\mathbb{F}$.
As a consequence of Theorem \ref{br},
we obtain the following
\begin{cor}
There is a natural map of sets
$$
P: \overline{H}_{BD-total}^{1} (sl(3,\mathbb{F}),r_{CG})\to\mathrm{Br}(\mathbb{F})
$$
\end{cor}
\begin{proof}
Following \cite{M2} or \cite{M3}, let us define a {\it quaternion algebra},
which has as a vector space over $\mathbb{F}$ the following basis: $\{\mathbf{1}, \mathbf{x},\mathbf{y},\mathbf{xy}\}$
subject to relations
$$
\mathbf{x}^2 =d\cdot\mathbf{1};\ \mathbf{y}^2 =b\cdot\mathbf{1};\ \mathbf{xy} =-\mathbf{yx}.
$$
This algebra is simple and thus, represents an element of $\mathrm{Br}(\mathbb{F})$, see for instance
\cite{M2} or \cite{M3} for a proof. In fact, it is easy to see that $P$ maps $\overline{H}_{BD-total}^{1} (sl(3,\mathbb{F}),r_{CG})$
into ${}_2\mathrm{Br}(\mathbb{F})$, the subgroup of elements of order $2$.
\end{proof}
\vskip0.2cm
\begin{rem}
The constructed algebra represents a non-trivial element of $\mathrm{Br}(\mathbb{F})$
if $b\notin {N_{\mathbb{F}(\sqrt{d})/\mathbb{F}}\ (\mathbb{F}(\sqrt{d}))^{*}}$ because in this case the Hilbert
symbol $(d,b)=-1$. If $b\in {N_{\mathbb{F}(\sqrt{d})/\mathbb{F}}\ (\mathbb{F}(\sqrt{d}))^{*}}$, the constructed
algebra is isomorphic to $\mathrm{Mat}(2,\mathbb{F})$ and thus, represents the unit of the Brauer group,
see \cite{M2}, \cite{M3}.
\end{rem}
\begin{rem}
Results of \cite{M1} imply that the set
$P(\overline{H}_{BD-total}^{1} (sl(3,\mathbb{F}),r_{CG}))$ generates ${}_2 \mathrm{Br}(\mathbb{F})$.
However, the map $P$ is not injective.
In \cite{M3} it was proved that $P(d,b)=P(m,n)$ if and only if quadratic forms $z^2 - dx^2 -by^2$
and   $z^2 - mx^2 -ny^2$ are equivalent over $\mathbb{F}$. It would be interesting to find a relation
between the corresponding Lie bialgebra structures on $sl(3,\mathbb{F})$.
\end{rem}

\begin{rem}
Similarly, for any $r_{BD}$ whose triple $(\Gamma_1,\Gamma_2,\tau)$
has only one string (without middlepoint) and satisfies
the condition of Proposition \ref{cond_tau}, one can construct a map
$$P:\overline{H}_{BD-total}^{1} (sl(n,\mathbb{F}),r_{BD})\longrightarrow
{}_2 \mathrm{Br}(\mathbb{F}).$$
\end{rem}

\textbf{Acknowledgment.} The authors are grateful to J. Brzezinski
for valuable discussions.


\begin{thebibliography}{13}

\bibitem {CG} Cremmer, E., and Gervais, J.-L.: The quantum group structure associated with non-linearly
extended Virasoro algebras. Comm. Math. Phys.\textbf{134}  (1990), 619--632.

\bibitem {BD} Belavin, A., Drinfeld, V.: Triangle equations and simple Lie
algebras. Soviet Sci. Rev. Sect. C: \textit{Math. Phys. Rev.} \textbf{4}
(1984) 93--165.

\bibitem{BZ} Benkart, G., Zelmanov E.: Lie algebras graded by finite root
systems and intersection matrix algebras. Invent. math. \textbf{126} (1996) 1--45.

\bibitem {D} Drinfeld, V. G.: Quantum groups. Proceedings ICM (Berkeley 1996)
\textbf{1} (1997), AMS, 798--820.

\bibitem {EK1} Etingof, P., Kazhdan, D.: Quantization of Lie bialgebras I.
Sel. Math. (NS)
\textbf{2} (1996) 1--41.

\bibitem {EK2} Etingof, P., Kazhdan, D.: Quantization of Lie bialgebras II.
Sel. Math.
(NS) \textbf{4} (1998) 213--232.

\bibitem {M1} Merkurjev, A.: On the norm residue symbol of degree 2. Dokladi Akad. Nauk SSSR
\textbf{261}  (1981), 542--547, Soviet Math. Doklady \textbf{24} (1981), 546-551.

\bibitem {M2} Milnor, J.: Introduction to algebraic K-theory. Princeton University Press (1971).

\bibitem {MSZ} Montaner, F., Stolin, A., Zelmanov, E.: Classification of
Lie bialgebras over current algebras. Sel. Math. (NS) \textbf{16} (2010)
935--962.

\bibitem {M3} O'Meara, O.: Introduction to quadratic forms. Springer (1963).

\bibitem{S2} Stolin, A.: Some remarks on Lie bialgebra structures on simple
complex Lie algebras. Comm. Alg. \textbf{27} (9) (1999) 4289--4302.

\bibitem{SP1} Stolin, A., Pop, I.: Classification of quantum groups and
Belavin--Drinfeld cohomologies. ArXiv Math. QA. 1303.4046v3.

\bibitem{SP2} Stolin, A., Pop, I.: On classification of quantum groups and Belavin--Drinfeld twisted cohomologies.
ArXiv Math. QA. 1309.7133v2.


\end{thebibliography}
\end{document}